\documentclass[11pt,leqno]{amsart}
\usepackage{verbatim,amssymb,amsfonts,latexsym,amsmath,amsthm,bbm,nicefrac,xspace,enumerate,tikz}

\newtheorem{theorem}{Theorem}[section]
\newtheorem*{maintheorem}{Main Theorem}

\newtheorem{lemma}[theorem]{Lemma}
\newtheorem{corollary}[theorem]{Corollary}

\newtheorem{question}[theorem]{Question}

\theoremstyle{definition}

\theoremstyle{remark}

\newcommand{\R}{\mathbb{R}}
\newcommand{\N}{\mathbb{N}}

\newcommand{\A}{\mathcal{A}}
\newcommand{\B}{\mathcal{B}}
\newcommand{\C}{\mathcal{C}}

\newcommand{\F}{\mathcal{F}}
\newcommand{\G}{\mathcal{G}}

\newcommand{\K}{\mathcal{K}}
\renewcommand{\P}{\mathcal{P}}

\newcommand{\explicitSet}[1]{\left\lbrace #1 \right\rbrace}

\newcommand{\set}[2]{\explicitSet{#1 \colon #2}}

\renewcommand{\a}{\alpha}
\renewcommand{\b}{\beta}

\newcommand{\dlt}{\delta}

\renewcommand{\k}{\kappa}
\newcommand{\s}{\sigma}

\newcommand{\w}{\omega}
\newcommand{\0}{\emptyset}

\newcommand{\sub}{\subseteq}
\newcommand{\rest}{\!\restriction\!}

\newcommand{\homeo}{\approx}

\newcommand{\cf}{\mathrm{cf}}
\newcommand{\card}[1]{\left\lvert #1 \right\rvert}

\newcommand{\continuum}{\mathfrak{c}}

\newcommand{\dom}{\mathfrak d}

\newcommand{\mad}{\mathfrak{a}}

\newcommand{\gch}{\ensuremath{\mathsf{GCH}}\xspace}
\newcommand{\zfc}{\ensuremath{\mathsf{ZFC}}\xspace}

\newcommand{\paar}{\mathfrak{par}}
\newcommand{\bspec}{\mathfrak{sp}(\text{\small Borel})}
\newcommand{\Spec}{\mathfrak{sp}(\text{\small closed})}
\newcommand{\specod}{\mathfrak{sp}(\text{\footnotesize OD($\R$)})}
\newcommand{\specg}{\mathfrak{sp}({\text{\footnotesize $\Gamma$}})}
\newcommand{\specgd}{\mathfrak{sp}({\text{\footnotesize $G_\dlt$}})}

\begin{document}

\title[The Borel partition spectrum at successors of singulars]{The Borel partition spectrum \linebreak at successors of singular cardinals}
\author{Will Brian}
\address {
W. R. Brian\\
Department of Mathematics and Statistics\\
University of North Carolina at Charlotte\\
9201 University City Blvd.\\
Charlotte, NC 28223, USA}
\email{wbrian.math@gmail.com}
\urladdr{wrbrian.wordpress.com}


\thanks{The author is supported by NSF grant DMS-2154229.}

\begin{abstract}
Assuming that $0^\dagger$ does not exist, we prove that if there is a partition of $\mathbb R$ into $\aleph_\omega$ Borel sets, then there is also a partition of $\mathbb R$ into $\aleph_{\omega+1}$ Borel sets.
\end{abstract}

\maketitle


\section{Introduction}

Define the \emph{Borel partition spectrum}, denoted $\bspec$, as follows:
$$\bspec = \set{|\P|}{\P \text{ is a partition of } \R \text{ into uncountably many Borel sets} }\!.$$
Let us begin by reviewing what is known about $\bspec$.

\vspace{.3mm}

\begin{enumerate}
\item $\aleph_1 \in \bspec$. 

\vspace{1mm}

\noindent Hausdorff showed in \cite{Hausdorff} that $2^\w$, and in fact any uncountable Polish space, can be expressed as an increasing union $\bigcup_{\xi < \w_1}E_\xi$ of $G_\dlt$ sets. 
This implies that $\R$, or any other uncountable Polish space, can be partitioned into $\aleph_1$ nonempty $F_{\s \dlt}$ sets.

\vspace{2mm}

\item $\continuum = \max\! \big( \bspec \big)$.

\vspace{1mm}

\noindent Clearly $\continuum \in \bspec$, because we may partition $\R$ into singletons. Assuming the Axiom of Choice (which we do throughout), $\R$ cannot be partitioned into $>\!|\R|$ sets.

\vspace{2mm}

\item $\bspec$ is closed under singular limits.

\vspace{1mm}

\noindent In other words, if $\lambda$ is singular and $\bspec \cap \lambda$ is unbounded in $\lambda$, then $\lambda \in \bspec$. This is proved in \cite[Theorem 2.8]{BBP} by adapting a (fairly straightforward) argument of Hechler \cite[Theorem 3.4]{Hechler} proving the analogous statement for the so-called {\small MAD} spectrum.
\end{enumerate}

\vspace{.3mm}

In addition to these positive \zfc-provable facts concerning $\bspec$, we also have several independence results, which seem to suggest that very little else can be proved from \zfc concerning $\bspec$.

\vspace{.3mm}

\begin{enumerate}
\item[$(4)$] It is consistent with arbitrary values of $\continuum$ that $\bspec = \{\aleph_1,\continuum\}$.

\vspace{1mm}

\noindent More precisely, if $\k^{\aleph_0} = \k$ and \gch holds up to $\k$, then forcing with $\mathrm{Fn}(\k,2)$ to add $\k$ mutually generic Cohen reals produces a model in which $\k = \continuum$ and $\bspec = \{\aleph_1,\k\}$. 
This is implicit in an argument of Miller \cite[Section 3]{Miller89}; a more general theorem from which this follows is proved explicitly by Blass in \cite{Blass}.

\vspace{2mm}

\item[$(5)$] It is consistent with arbitrary values of $\continuum$ that $\bspec = [\aleph_1,\continuum]$.

\vspace{1mm}

\noindent More precisely, given any $\k \geq \continuum$ with $\k^{\aleph_0} = \k$, there is a ccc forcing extension in which $\continuum = \k$ and $\bspec = [\aleph_1,\continuum]$. This follows from \cite[Theorem 3.11]{Brian&Miller}, where the result is proved for $2^\w$ instead of $\R$.
\end{enumerate}

\vspace{.3mm}

\noindent In other words, $(4)$ and $(5)$ say that, regardless of what $\continuum$ may be, $\bspec$ can be as small as possible or as large as possible (given the requirements in $(1)$ and $(2)$). Furthermore, the next result shows that $\bspec$ can be made to look chaotic and patternless, equal to an almost (but not quite) arbitrary set of cardinals satisfying $(1)$, $(2)$, and $(3)$.

\vspace{.3mm}

\begin{enumerate}
\item[$(6)$] Suppose $C$ is a set of uncountable cardinals such that:
\begin{enumerate}
\item $C$ is countable,
\item $\aleph_1 \in C$,
\item $C$ has a maximum, and $\max (C)^{\aleph_0} = \max (C)$,
\item $C$ is closed under singular limits, and
\item if $\lambda \in C$ and $\cf(\lambda) = \w$, then $\lambda^+ \in C$.
\end{enumerate}
\noindent Assuming \gch holds up to $\max C$, there is a ccc forcing extension in which $C = \bspec$.

\vspace{1.3mm}

\noindent This is proved as Corollary 3.3 in \cite{BBP}. 
An immediate consequence of this (Corollary 3.4 in \cite{BBP}) is: for any $A \sub \w \setminus \{0\}$ with $1 \in A$, there is a forcing extension in which $\bspec = \set{\aleph_n}{n \in A} \cup \{\aleph_\w,\aleph_{\w+1}\}$.
\end{enumerate}

\vspace{.3mm}

Of the requirements given for $C$ in $(6)$, items (b), (c), and (d) correspond directly to items $(1)$, $(2)$, and $(3)$ above. These requirements for $C$ cannot be eliminated from $(6)$, because they are necessary features of $\bspec$ in any forcing extension. 
Item (a), on the other hand, does not represent a necessary feature of $\bspec$; we know this because $(5)$ implies that $\bspec$ can be uncountable. Item (a) is just an artifact of the proof of $(6)$ in \cite{BBP}, and it is conceivable that a more careful or more clever proof could, at some point in the future, eliminate it from the statement of $(6)$ altogether.

This leaves item (e), which forms the topic of this paper. 
Our main result is that, in a sufficiently ``$L$-like'' set-theoretic universe, (e) represents a necessary feature of $\bspec$. 
Thus, in such a universe there is nontrivial structure to $\bspec$ beyond what is stated in items $(1)$, $(2)$, and $(3)$ above.

\begin{maintheorem}
Suppose $0^\dagger$ does not exist. If $\k$ is a singular cardinal with $\cf(\k) = \w$ and $\k \in \mathfrak{sp}(\text{\emph{\small Borel}})$, then $\k^+ \in \mathfrak{sp}(\text{\emph{\small Borel}})$.
\end{maintheorem}

\noindent In particular, assuming that $0^\dagger$ does not exist, if there is a partition of $\mathbb R$ into $\aleph_\omega$ Borel sets then there is also a partition of $\mathbb R$ into $\aleph_{\omega+1}$ Borel sets.

The assertion ``$0^\dagger$ exists'' is a large cardinal axiom a little stronger (in consistency strength) than the existence of a measurable cardinal. 
Roughly, $0^\dagger$ bears the same relationship to the $L$-like models of the form $L[\mu]$, where $\mu$ is a normal measure on a measurable cardinal, that $0^\sharp$ bears to $L$.
(See \cite{Kanamori} for more.)
In particular, if $0^\dagger$ exists then there is an inner model containing a measurable cardinal. 
Thus, by contraposition: if $\aleph_\w \in \bspec$ while $\aleph_{\w+1} \notin \bspec$, there is an inner model containing a measurable cardinal.

We present two proofs of the main theorem. Both of them use the assertion ``$0^\dagger$ does not exist'' in the guise of cofinal Kurepa families. What these are, and their connection to $0^\dagger$, is explained in Section~\ref{sec:Kurepa}.
Then in Sections~\ref{sec:main} and \ref{sec:encore} we present our two proofs.
These proofs use cofinal Kurepa families in different but clearly analogous ways.
The proof in Section~\ref{sec:main} is more direct, optimized for both length and clarity. However, the proof in Section~\ref{sec:encore} gives slightly more in the end, and it also connects the topic of this paper to an old Scottish Book problem of Stefan Banach.
\section{Preliminaries: cofinal Kurepa families}\label{sec:Kurepa}

Given an infinite set $A$, recall that $[A]^\w$ denotes the set of all countably infinite subsets of $A$.

\begin{itemize}
\item[$\circ$] $\F \sub [A]^\w$ is \emph{cofinal} in $[A]^\w$ if for every $X \in [A]^\w$ there is some $Y \in \F$ such that $Y \supseteq X$. In other words, $\F \sub [A]^\w$ is cofinal if it is cofinal (in the usual sense of the word) in the poset $\hspace{-.1mm} \big( [A]^{\mu},\sub \hspace{-1mm} \big)$.
\end{itemize}
\begin{itemize}
\item[$\circ$] $\F \sub [A]^\w$ is \emph{Kurepa} if $\set{X \cap Y}{Y \in \F}$ is countable for every countable set $X$.
\end{itemize}

For example, $\F = [A]^\w$ is cofinal in $[A]^\w$ but not Kurepa, for any infinite set $A$. If $A$ is uncountable, any countable $\F \sub [A]^\w$ is Kurepa but not cofinal.
Generally, if $\F$ is cofinal then so is every $\G \supseteq \F$, and if $\F$ is Kurepa then so is every $\G \sub \F$. 
Thus cofinal families are ``large'' and Kurepa families are ``small'' but in different senses. A subset of $[A]^\w$ is a \emph{cofinal Kurepa family} in $[A]^\w$ if it is both cofinal in $[A]^\w$ and Kurepa.

For example, $\set{\a}{\w \leq \a < \w_1}$ is a cofinal Kurepa family in $[\w_1]^\w$. (Here, as usual, we adopt the convention that an ordinal is equal to the set of its predecessors.)
Note that the (non-)existence of cofinal Kurepa families in $[A]^\w$ depends only on the cardinality of $A$, so in fact this shows there is a cofinal Kurepa family in $[A]^\w$ for every set $A$ of cardinality $\aleph_1$.
It turns out that the same is true for every cardinal below $\aleph_\w$:

\begin{theorem}\label{thm:smallKFams}
If $|A| < \aleph_\w$ then there is a cofinal Kurepa family in $[A]^\w$. 
\end{theorem}
\begin{proof}
See \cite[Corollary 7.6.22]{Todorcevic}.
\end{proof}

The situation is more subtle for cardinals $\geq \aleph_\w$, as described in the following two theorems.

\begin{theorem}\label{thm:KurepaExistence}
Suppose $\square_\k$ holds for every uncountable cardinal $\k$ with $\cf(\k) = \w$. Then there is a cofinal Kurepa family  in $[A]^\w$ for every set $A$.
\end{theorem}
\begin{proof}
See \cite[Theorem 7.6.26]{Todorcevic} and the comments following its proof.
\end{proof}

\begin{theorem}\label{thm:KurepaNonExistence}
It is consistent relative to a huge cardinal that there is no cofinal Kurepa family in $[\w_\w]^\w$.
\end{theorem}
\begin{proof}
The generalized Chang Conjecture $(\aleph_{\w+1},\aleph_\w) \to\hspace{-3.14mm}\to (\aleph_1,\aleph_\w)$ implies there is no cofinal Kurepa family on $\w_\w$. (This is well known, and is in fact a relatively straightforward consequence of $(\aleph_{\w+1},\aleph_\w) \to\hspace{-3.14mm}\to (\aleph_1,\aleph_\w)$. It is mentioned without proof in \cite[Section 7.6]{Todorcevic}, and an indirect proof can be found in \cite[Section 3]{KMS}.)
The consistency of $(\aleph_{\w+1},\aleph_\w) \to\hspace{-3.14mm}\to (\aleph_1,\aleph_\w)$ was first proved in \cite{MLS} from a large cardinal hypothesis a little stronger than the existence of a huge cardinal.
Later improvements in \cite{EH} reduced the necessary hypothesis to ``just'' a huge cardinal.
\end{proof}

For every set $A$, define
$$\mathrm{cf} [A]^{\mu} \,=\, \min \set{|\F|}{\F \text{ is cofinal in $[A]^\w$}}.$$
Note that $\cf[A]^\w$ depends only on the cardinality of $A$.

If $\F$ is cofinal in $[A]^\w$ then $\bigcup \F = A$, so $|A| = \card{\bigcup \F} \leq \aleph_0 \cdot |\F|$.
If $A$ is uncountable, this implies $|\F| \geq |A|$.
Thus $\cf[\k]^\w \geq \k$ for all uncountable cardinals $\k$.
And if $\cf(\k) = \w$, we get an even stronger bound: a diagonalization argument shows in this case that $\cf [\k]^\w \geq \k^+$.

\emph{Shelah's Strong Hypothesis}, abbreviated $\mathsf{SSH}$, is the following statement.
\begin{itemize}
\item[$\mathsf{SSH}:$] For every uncountable cardinal $\k$,
$$\cf[\k]^\w \,=\, \begin{cases}
\k &\text{if } \cf(\k) > \w, \\
\k^+ &\text{if } \cf(\k) = \w.
\end{cases}$$ \end{itemize}
In other words, $\mathsf{SSH}$ says that $\cf [\k]^\w$ should be as small as possible (subject to the restrictions in the previous paragraph) for every uncountable cardinal $\k$.
The original phrasing of Shelah’s Strong Hypothesis by Shelah is somewhat different: it involves the pseudo-power function from pcf theory. 
The statement here labeled $\mathsf{SSH}$ is equivalent to the original form of Shelah’s Strong Hypothesis by a theorem of
Matet \cite{Matet}. 

\begin{theorem}\label{thm:KFams,Yo}
Suppose $0^\dagger$ does not exist. Then $\mathsf{SSH}$ holds, and for every set $A$ there is a cofinal Kurepa family in $[A]^\w$.
\end{theorem}
\begin{proof}
By work in \cite{JensenSteel} and \cite{Sch&Zeeman}, if $\square_\k$ fails for a singular cardinal $\k$, then there is an inner model for a Woodin cardinal. This implies ``$0^\dagger$ exists'' (with plenty of room to spare). 
By contraposition, if $0^\dagger$ does not exist then $\square_\k$ holds for every singular $\k$. 
By Theorem~\ref{thm:KurepaExistence}, this implies there is a cofinal Kurepa family in $[A]^\w$ for every set $A$.

It is worth mentioning that the failure of $\square_\k$ for singular $\k$ is possibly much stronger than the existence of a Woodin cardinal. (This is already known to be true for $\k > \continuum$: see \cite[Corollary 5]{Sch&Zeeman}. But in the context of this paper, we are more concerned with $\square_\k$ for $\k < \continuum$.) To quote \cite{Sch&Zeeman}, ``the extent to which we can $\dots$ obtain lower bounds on the large cardinal consistency strength of the failure of $\square_\k$ $\dots$ is limited by our ability to construct core models and prove covering theorems for them.''

The bottleneck for this theorem is not the existence of cofinal Kurepa families, which follows from $\square_\k$ holding for singular $\k$, but rather $\mathsf{SSH}$. The exact consistency strength of $\neg\mathsf{SSH}$ is not known. The failure of $\mathsf{SSH}$ at some $\k > \continuum$ is equivalent to the failure of the Singular Cardinals Hypothesis, $\mathsf{SCH}$. The exact consistency strength of this is known, due to work of Gitik \cite{Gitik1,Gitik2}: it is equiconsistent with the existence of a measurable cardinal $\mu$ with Mitchell order $\mu^{++}$.

Gitik's work makes use of the so-called weak covering lemma holding over inner models of \gch. The weak covering lemma does not seem to tell us anything about $\mathsf{SSH}$ for cardinals $\k < \continuum$. However, $\mathsf{SSH}$ follows from the full covering lemma holding over an inner model of \gch. (This was first observed in \cite[Lemmas 4.9 and 4.10]{JK}.)

By work of Dodd and Jensen \cite{DJ}, if there is no inner model containing a
measurable cardinal, then the covering lemma holds over an inner model
of \gch (specifically, the Dodd-Jensen core model $K^\mathrm{DJ}$). By further work
of Dodd and Jensen \cite{DJ2}, if there is an inner model containing a measurable
cardinal but $0^\dagger$ does not exist, then the covering lemma holds over an inner
model of \gch (specifically, either $L[\mu]$ for some measure $\mu$ with $\mathrm{crit}(\mu)$ as
small as possible, or $L[\mu, C]$ for some sequence $C$ Prikry-generic over $L[\mu]$). For more
information, see Mitchell’s article \cite{Mitchell}.

Either way, if $0^\dagger$ does not exist then Jensen’s covering lemma holds over
an inner model M of \gch, and $\mathsf{SCH}$ follows by the results in \cite{JK}.
\end{proof}

To close this section, we include one more fact concerning cofinal Kurepa families that will be used in what follows.

\begin{lemma}\label{lem:Ksize}
Every cofinal Kurepa family in $[A]^\w$ has cardinality $\cf[A]^\w$.
\end{lemma}
\begin{proof}
Suppose $\F$ is a cofinal Kurepa family in $[A]^\w$. 
Because $\F$ is cofinal, $|\F| \geq \cf[A]^\w$. 
To prove the reverse inequality, fix some cofinal $\C \sub [A]^\w$ with $|\C| = \cf[A]^\w$.

For every $X \in \F$, choose a set $c(X) \in \C$ with $c(X) \supseteq X$ (which can be done because $\C$ is cofinal).
If $|\C| = \cf[A]^\w < |\F|$, then
by the pigeonhole principle there is some particular $Y \in \C$ such that $\set{X \in \F}{c(X) = Y}$ is uncountable. But then $\set{X \cap Y}{X \in \F} \supseteq \set{X \in \F}{c(X) = Y}$ is uncountable, contradicting our assumption that $\F$ is Kurepa.
\end{proof}

\begin{corollary}\label{cor:KFams}
Assuming $0^\dagger$ does not exist, if $\k$ is an uncountable cardinal with $\cf(\k) = \w$ and $|A| = \k$, then there is a cofinal Kurepa family in $[A]^\w$ with cardinality $\k^+$.
\end{corollary}
\begin{proof}
This follows immediately from Theorem~\ref{thm:KFams,Yo} and Lemma~\ref{lem:Ksize}.
\end{proof}

\section{A proof of the main theorem}\label{sec:main}

We begin this section by observing that the definition of $\bspec$ does not depend on $\R$: it remains unchanged when $\R$ is replaced by any other uncountable Polish space. This is proved already as Theorem 2.1 in \cite{BBP}, but the proof is short, and is used in the proof of our main theorem, so we reproduce the argument here.

\begin{lemma}\label{lem:bspec}
If $X$ is any uncountable Polish space, then
$$\mathfrak{sp}(\text{\emph{\small Borel}}) = \set{\k > \aleph_0}{\text{there is a partition of } X \text{ into } \k \text{ Borel sets}}.$$
\end{lemma}
\begin{proof}
By a theorem of Kuratowski (see \cite[Theorem 15.6]{Kechris}), any two uncountable Polish spaces are Borel isomorphic: in other words, there is a bijection $f: \R \to X$ such that $A \sub \R$ is Borel if and only if $f[A]$ is Borel. Thus if $\P$ is any partition of $\R$ into Borel sets, then $\set{f[B]}{B \in \P}$ is a partition of $X$ into Borel sets, and if $\mathcal Q$ is any partition of $X$ into Borel sets, then $\set{f^{-1}[B]}{B \in \mathcal Q}$ is a partition of $\R$ into Borel sets.
\end{proof}

\begin{theorem}\label{thm:K}
Let $\k$ be a cardinal and suppose there is a cofinal Kurepa family in $[\k]^\w$. If $\k \in \mathfrak{sp}(\text{\emph{\small Borel}})$, then $\cf[\k]^\w \in \mathfrak{sp}(\text{\emph{\small Borel}})$.
\end{theorem}
\begin{proof}
Suppose $\k \in \bspec$, and let $\P$ be a partition of $\R$ into $\k$ Borel sets. Because there is a cofinal Kurepa family in $[\k]^\w$ and $|\P| = \k$, there is a cofinal Kurepa family in $[\P]^\w$. Let $\K$ be some such family.

To prove the theorem, it suffices (by Lemmas~\ref{lem:bspec} and \ref{lem:Ksize}) to show that there is a partition of $\R^\w$ into $\card{\K}$ Borel sets. We adopt the convention that points in $\R^\w$ are functions $\w \to \R$, so that $x(n)$ denotes the $n^{\mathrm{th}}$ coordinate of $x$ for any given $x \in \R^\w$.

For every $\A \in \K$, define
$$X_\A = (\textstyle \bigcup \A)^\w = \set{x \in \R^\w}{x(n) \in \textstyle \bigcup \A \text{ for all }n \in \w}.$$
Note that, because each $\A \in \K$ is a countable collection of Borel subsets of $\R$, each $X_\A$ is a Borel subset of $\R^\w$.

We claim that $\bigcup \set{X_\A}{\A \in \K} = \R^\w$. To see this, let $x \in \R^\w$. For each $n \in \w$, there is some $B_n \in \P$ with $x \in B_n$. Because $\K$ is cofinal in $[\P]^\w$, there is some $\A \in \K$ with $\set{B_n}{n \in \w} \sub \A$. This implies $x \in X_\A$.

Let $\prec$ be a well ordering of $\K$, and for each $\A \in \K$ define
$$Y_\A = X_\A \setminus \textstyle \bigcup \set{X_\B}{\B \prec \A}.$$
Because $\bigcup \set{X_\A}{\A \in \K} = \R^\w$, the $Y_\A$'s form a partition of $\R^\w$, or more precisely (as some of the $Y_\A$'s may be empty), $\mathcal Q = \set{Y_\A}{\A \in \K} \setminus \{\0\}$ is a partition of $\R^\w$. To prove the theorem, it remains to show that $\card{\mathcal Q} = |\K|$, and that every $Y_\A$ is a Borel subset of $\R^\w$.

To see that $\card{\mathcal Q} = |\K|$, it suffices to show that $Y_\A \neq \0$ for $|\K|$-many $\A$. (This is because it is clear from the definition that $Y_\A \cap Y_\B = \0$ whenever $\A \neq \B$; the only non-injectivity in the map $\A \mapsto Y_\A$ is that it may map many sets to $\0$.) For every $\A \in \K$, define
$$\mathcal Z_\A = \set{B \in \P}{x(n) \in B \text{ for some } x \in Y_\A \text{ and some } n \in \w}.$$
Note that $\mathcal Z_\A = \0$ whenever $Y_\A = \0$. 
For every $\A \in \K$, we have $Y_\A \sub X_\A$, which means that
$$\mathcal Z_\A \sub \set{B \in \P}{x(n) \in B \text{ for some } x \in X_\A \text{ and some } n \in \w} = \A.$$
In particular, $\set{\mathcal Z_\A}{\A \in \K}$ is a collection of countable subsets of $\P$. 
We claim this collection is cofinal in $[\P]^\w$. To see this, let $\set{B_n}{n \in \w}$ be an arbitrary countable subset of $\P$, and for each $n \in \w$ fix some $x_n \in B_n$. 
Consider the point $x \in \R^\w$ with $x(n) = x_n$ for all $n \in \w$. Because $\mathcal Q$ is a partition of $\R^\w$, there is some $\A \in \K$ with $x \in Y_\A$. It follows that $\set{B_n}{n \in \w} \sub \mathcal Z_\A$.
As $\set{B_n}{n \in \w}$ was an arbitrary member of $[\P]^\w$, this shows $\set{\mathcal Z_\A}{\A \in \K}$ is cofinal in $[\P]^\w$. 
Hence $\card{\set{\mathcal Z_\A}{\A \in \K}} \geq \cf[\P]^\w = |\K|$, which means $\card{\set{\mathcal Z_\A}{\A \in \K}} = |\K|$.
In particular, $\mathcal Z_\A \neq \0$ for $|\K|$-many $\A \in \K$, 
which implies $Y_\A \neq \0$ for $|\K|$-many $\A \in \K$. 
Hence $|\mathcal Q| = |\K|$.

To see that each $Y_\A$ is Borel in $\R^\w$, we use the Kurepa property of $\K$. 
Fix $\A \in \K$. 
Because $\K$ is Kurepa, there is a countable $\F \sub \K$ such that 
$\set{\A \cap \B}{\B \in \K} = \set{\A \cap \B}{\B \in \F}$. 
This implies 
$$\set{\A \cap \B}{\B \in \K \text{ and } \B \prec \A} = \set{\A \cap \B}{\B \in \G}$$
for some (countable) $\G \sub \F$. 
In particular, 
\begin{align*}
Y_\A &=\textstyle  X_\A \setminus \bigcup \set{ X_\B }{\B \prec \A} \\
&=\textstyle (\bigcup \A)^\w \setminus \bigcup \set{ (\bigcup \B)^\w }{\B \prec \A} \\
&=\textstyle (\bigcup \A)^\w \setminus \bigcup \set{ (\bigcup (\B \cap \A))^\w }{\B \prec \A} \\
&=\textstyle (\bigcup \A)^\w \setminus \bigcup \set{ (\bigcup (\B \cap \A))^\w }{\B \in \G} \\
&=\textstyle X_\A \setminus \bigcup \set{X_\B}{\B \in \G}.
\end{align*}
Because $X_\A$ and all of the $X_\B$'s are Borel sets in $\R^\w$, 
and because $\G$ is countable, it follows that $Y_\A$ is Borel.
\end{proof}

\begin{theorem}\label{thm:main}
Assume $0^\dagger$ does not exist.
If $\k$ is a singular cardinal with $\cf(\k) = \w$ and $\k \in \mathfrak{sp}(\text{\emph{\small Borel}})$, then $\k^+ \in \mathfrak{sp}(\text{\emph{\small Borel}})$.
\end{theorem}
\begin{proof}
This follows immediately from Corollary~\ref{cor:KFams} and Theorem~\ref{thm:K}.
\end{proof}

This completes our first proof of the main theorem. 
Note that the hypothesis ``$0^\dagger$ does not exist'' can be replaced with the possibly weaker hypothesis ``$\cf[\k]^\w = \k^+$ and there is a cofinal Kurepa family in $[\k]^\w$.''

\begin{question}
Is it consistent, relative to some large cardinal hypothesis, that there is a singular cardinal $\lambda$ with $\cf(\lambda) = \w$ such that $\lambda \in \mathfrak{sp}(\text{\emph{\small Borel}})$ but $\lambda^+ \notin \mathfrak{sp}(\text{\emph{\small Borel}})$?
\end{question}

\begin{question}
In particular, is it consistent relative to some large cardinal hypothesis that $\aleph_\w \in \mathfrak{sp}(\text{\emph{\small Borel}})$ but $\aleph_{\w+1} \notin \mathfrak{sp}(\text{\emph{\small Borel}})$?
\end{question}

\begin{question}
Is it consistent, relative to some large cardinal hypothesis, that there is a singular cardinal $\lambda$ with $\cf(\lambda) = \w$ such that $\lambda \in \mathfrak{sp}(\text{\emph{\small Borel}})$ but $\cf[\lambda]^\w \notin \mathfrak{sp}(\text{\emph{\small Borel}})$?
\end{question}

\section{A second proof of the main theorem}\label{sec:encore}

In this section we give our second proof of the main theorem. This second proof has a more topological flavor than the one presented in the previous section. Throughout this section, all ordinals are considered to have the discrete topology. Thus, for example, $\k^\w$ is a completely metrizable space, sometimes called the Baire space of weight $\k$ (e.g., in \cite[Example 4.2.12]{Engelking}). 

Given a topological space $X$, define
$$\paar(X) = \min \set{|\P|}{\P \text{ is a partition of } X \text{ into Polish spaces} }.$$
Note that $\paar(X)$ is well-defined and $\leq\!|X|$ for every space $X$, because we may partition $X$ into singletons.
The invariant $\paar(X)$ was introduced and studied in \cite{BCP}, the main idea being to compare $\paar(X)$ with $\mathfrak{cov}(X)$ (defined as the least size of a covering of $X$ with Polish spaces) when $X$ is completely metrizable. 

The next two lemmas, which are implicit in \cite{BCB} and \cite{Brian&Miller}, allow us to transfer facts about the spaces $\k^\w$ and $\paar(\k^\w)$ to facts about $\bspec$.

\begin{lemma}\label{lem:pushforward}
If there is a continuous bijection $X \to \R$, for some space $X$, then there is a partition of $\R$ into $\paar(X)$ Borel sets.
\end{lemma}
\begin{proof}
Suppose $f: X \to \R$ is a continuous bijection, and let $\P$ be a partition of $X$ into $\paar(X)$ Polish spaces. 
By a theorem of Lusin and Suslin \cite[Theorem 15.1]{Kechris}, 
if $A$ is Polish and $g: A \to \R$ is a continuous injection, then $g[A]$ is Borel in $\R$ (in fact, $g[B]$ is Borel in $\R$ for every Borel $B \sub A$). 
In particular, $f \rest A$ is a continuous injection $A \to \R$ for every $A \in \P$, which means $f[A]$ is Borel in $\R$ for every $A \in \P$. Hence $\set{f[A]}{A \in \P}$ is a partition of $\R$ into Borel sets.
\end{proof}

\begin{lemma}\label{lem:steppingup}
Let $\k$ be an infinite cardinal. If there is a partition of $\R$ into $\k$ Borel sets, then there is a continuous bijection $\k^\w \to \R$.
\end{lemma}
\begin{proof}
Suppose $\k \in \bspec$. By Lemma~\ref{lem:bspec}, this implies there is a partition $\P$ of the Baire space $\w^\w$ into $\k$ Borel sets. 

By a theorem of Lusin and Suslin \cite[Theorem 13.7]{Kechris}, every Borel subset of a Polish space is the bijective continuous image of a closed subset of $\w^\w$. In particular, for every $A \in \P$, there is a closed $F_A \sub \w^\w$ and a continuous bijection $f_A: F_A \to A$. 
This implies that the map $(x,y) \mapsto (f_A(x),y)$ is a continuous bijection $F_A \times \w^\w \to A \times \w^\w$. 
By the Alexandroff-Urysohn characterization of the Baire space $\w^\w$ \cite[Exercise 7.2.G]{Engelking}, $F_A \times \w^\w \homeo \w^\w$.
Thus for every $A \in \P$, there is a continuous bijection $g_A: \w^\w \to A \times \w^\w$.

Let $\mathcal Q = \set{A \times \w^\w}{A \in \P}$; this is a partition of $\w^\w \times \w^\w$ into $\k$ Borel sets.
Taking the disjoint union of the mappings $g_A$ from the previous paragraph, we obtain a continuous bijection $G: \k \times \w^\w \to \w^\w \times \w^\w$.

Let $G^\w$ denote the mapping $(x_0,x_1,x_2,\dots) \mapsto (G(x_0),G(x_1),G(x_2),\dots)$ (sometimes called the ``diagonal mapping''). 
Because $G$ is a continuous bijection $\k \times \w^\w \to \w^\w \times \w^\w$,
the diagonal mapping $G^\w$ is a continuous bijection $(\k \times \w^\w)^\w \to (\w^\w \times \w^\w)^\w$. But $(\k \times \w^\w)^\w \homeo \k^\w$ and $(\w^\w \times \w^\w)^\w \homeo \w^\w$ (both of these facts follow from \cite[Exercise 7.2.G]{Engelking}), so this shows there is a continuous bijection $\k^\w \to \w^\w$.

To finish the proof, simply compose this continuous bijection $\k^\w \to \w^\w$ with a continuous bijection $\w^\w \to \R$. (Recall that every Polish space without isolated points is a continuous bijective image of $\w^\w$ \cite[Exercise 7.15]{Kechris}.)
\end{proof}

\begin{theorem}\label{thm:nearlythere}
Let $\k$ be an uncountable cardinal. If $\k \in \bspec$, then $\paar(\k^\w) \in \bspec$ also.
\end{theorem}
\begin{proof}
This follows immediately from Lemmas~\ref{lem:pushforward} and \ref{lem:steppingup}.
\end{proof}

Note the similarity between Theorem~\ref{thm:K} and Theorem~\ref{thm:nearlythere}. 
Our next theorem shows that $\paar(\k^\w) = \cf[\k]^\w$ if there is a cofinal Kurepa family on $\k$, which will complete our second proof of Theorem~\ref{thm:K} and the main theorem (which follows directly from it).
We note that the equality $\paar(\k^\w) = \cf[\k]^\w$ need not be true in general: it is consistent relative to a huge cardinal to have $\paar(\k^\w) > \cf[\k]^\w$ (see \cite[Section 3]{BCP}).

\begin{theorem}\label{thm:par}
Let $\k$ be an uncountable cardinal. If there is a cofinal Kurepa family on $\k$, then $\paar(\k^\w) = \cf[\k]^\w$.
\end{theorem}

The following proof contains some ideas similar to what is found in \cite[Sections 2 and 4]{BCP}, where
what is called Corollary\ref{cor:ParAlephOmega} below was first proved.
Something similar to the proof of Theorem~\ref{thm:par} is implicit in \cite{BCP}, though Jensen matrices are used there rather than Kurepa families.
However, the arguments in \cite{BCP} are aimed at proving something rather different, and they feature several ideas that are not relevant here. 
What follows is a relatively short proof, modifying and distilling from \cite{BCP} what is needed for our analysis of $\bspec$. 

\begin{proof}[Proof of Theorem~\ref{thm:par}]
Let $\K$ be a cofinal Kurepa family on $\k$.
To prove the theorem, it suffices, by Lemma~\ref{lem:Ksize}, to show there is a partition of $\k^\w$ into $\card{\K}$ Polish spaces. Recall that a subspace of a completely metrizable space is itself completely metrizable if and only if it is $G_\dlt$. Thus a subspace of $\k^\w$ is Polish if and only if it is both second countable and $G_\dlt$.

For every $A \in \K$, note that $A^\w$ is second countable (because $A$ is countable) and closed in $\k^\w$.
Furthermore, we claim $\k^\w$ is covered by the sets of this form: i.e., $\bigcup \set{A^\w}{A \in \K} = \k^\w$. 
To see this, let $x \in \k^\w$. As $\K$ is cofinal in $[\k]^\w$, there is some $A \in \K$ with $\set{x(n)}{n \in \w} \sub A$, and this means $x \in A^\w$. As $x$ was arbitrary, $\bigcup \set{A^\w}{A \in \K} = \k^\w$.

Let $\prec$ be a well ordering of $\K$, and for each $A \in \K$ define
$$Y_A = A^\w \setminus \textstyle \bigcup \set{B^\w}{B \prec A}.$$
Because $\bigcup \set{A^\w}{A \in \K} = \k^\w$, the $Y_A$'s form a partition of $\k^\w$, or more precisely (because some of the $Y_A$ may be empty), $\mathcal Q = \set{Y_A}{A \in \K} \setminus \{\0\}$ is a partition of $\k^\w$. To prove the theorem, it remains to show that $\card{\mathcal Q} = |\K|$, and that every $Y_A \in \mathcal Q$ is Polish.

To see that $\card{\mathcal Q} = \cf[\k]^\w$, define for every $A \in \K$
$$\mathcal Z_A = \set{\a \in \k}{x(n) = \a \text{ for some } x \in Y_A \text{ and some } n \in \w}.$$
Note that we have $\mathcal Z_A = \0$ whenever $Y_A = \0$. 
For every $A \in \K$, we have $Y_A \sub X_A$, which means that
$$\mathcal Z_A \sub \set{\a \in \k}{x(n) = \a \text{ for some } x \in A^\w \text{ and some } n \in \w} = A.$$
In particular, $\set{\mathcal Z_A}{A \in \K}$ is a collection of countable subsets of $\k$. 
We claim this collection is cofinal in $[\k]^\w$. To see this, let $\set{x_n}{n \in \w}$ be an arbitrary countable subset of $\k$, and define $x \in \k^\w$ by setting $x(n) = x_n$ for all $n \in \w$. Because $\mathcal Q$ is a partition of $\k^\w$, there is some $A \in \K$ with $x \in Y_A$. It follows that $\set{x_n}{n \in \w} \sub \mathcal Z_A$.
As $\set{x_n}{n \in \w}$ was arbitrary, it follows that $\set{\mathcal Z_A}{A \in \K}$ is cofinal in $[\k]^\w$. 
Hence $\card{\set{\mathcal Z_A}{A \in \K}} \geq \cf[\k]^\w = |\K|$.
In particular, $\mathcal Z_A \neq \0$ for $|\K|$-many $A \in \K$. 
Consequently, $Y_A \neq \0$ for $|\K|$-many $A \in \K$.
As the $Y_A$ are disjoint, this implies $\card{\mathcal Q} = |\K|$.

To see that each $Y_A$ is Polish, we use the fact that $\K$ is Kurepa. 
Fix $A \in \K$. 
Because $\K$ is Kurepa, there is a countable $\F \sub \K$ such that 
$\set{A \cap B}{B \in \K} = \set{A \cap B}{B \in \F}$. 
This implies 
$$\set{A \cap B}{B \in \K \text{ and } B \prec A} = \set{A \cap B}{B \in \G}$$
for some (countable) $\G \sub \F$. 
In particular, 
\begin{align*}
Y_A &=\textstyle  A^\w \setminus \bigcup \set{ B^\w }{B \prec A} \\
&=\textstyle A^\w \setminus \bigcup \set{ (B \cap A)^\w }{B \prec A} \\
&=\textstyle A^\w \setminus \bigcup \set{ (B \cap A)^\w }{B \in \G} \\
&=\textstyle A^\w \setminus \bigcup \set{B^\w}{B \in \G}.
\end{align*}
Because $A^\w$ is second countable and closed, and each of the $B^\w$ is also closed, this shows $Y_A$ is second countable and $G_\dlt$, i.e., Polish.
\end{proof}

\begin{corollary}\label{cor:ParAlephOmega}
If $0 < n < \w$, then $\paar(\w_n^\w) = \aleph_n$.
Furthermore, if $0^\dagger$ does not exist then for every cardinal $\k > \aleph_0$,
$$\paar(\k^\w) = \cf[\k]^\w = \begin{cases}
\k &\text{if } \cf(\k) > \w, \\
\k^+ &\text{if } \cf(\k) = \w.
\end{cases}$$
\end{corollary}
\begin{proof}
This follows from Corollary~\ref{cor:KFams} and Theorems~\ref{thm:smallKFams} and \ref{thm:par}.
\end{proof}

\begin{theorem}\label{thm:main2}
Assume $0^\dagger$ does not exist, and let $\k$ be a singular cardinal with $\cf(\k) = \w$. If $\k \in \mathfrak{sp}(\text{\emph{\small Borel}})$, then $\k^+ \in \mathfrak{sp}(\text{\emph{\small Borel}})$ also.
\end{theorem}
\begin{proof}
This follows immediately from Theorem~\ref{thm:nearlythere} and Corollary~\ref{cor:ParAlephOmega}.
\end{proof}

This completes the second proof of our main theorem.
To end this section, we point out how the results from this section relate to what is sometimes called the \emph{Banach problem}, posed by Stefan Banach in the Scottish Book:
\begin{itemize}
\item[$\!$] When does a metric space admit a continuous bijective mapping onto a compact metric space?
\end{itemize}

\begin{theorem}\label{thm:Banach}
Let $\k$ be an uncountable cardinal with $\paar(\k^\w) = \k$. Then the following are equivalent: 
\begin{enumerate}
\item $\k \in \mathfrak{sp}(\text{\emph{\small Borel}})$.
\item There is a continuous bijective mapping from $\k^\w$ onto $\R$.
\item There is a continuous bijective mapping from $\k^\w$ onto a compact metric space.
\end{enumerate}
\end{theorem}
\begin{proof}
$(1) \Rightarrow (2)$ was proved in Lemma~\ref{lem:steppingup}. $(2) \Rightarrow (3)$ because there is a continuous bijection from $\R$ to a compact metric space: for example, it is easy to describe a continuous bijection from $\R$ to a wedge sum of two circles (an ``8'').
Finally, $(3) \Rightarrow (1)$ by the proof of Lemma~\ref{lem:pushforward}, as the role of $\R$ in that proof can be played equally well by any compact metric space.
\end{proof}

By Theorem~\ref{thm:par}, the hypothesis of this theorem is satisfied by $\k = \aleph_n$ for any $n < \w$; or if $0^\dagger$ does not exist, it is satisfied by any cardinal $\k$ with $\cf(\k) > \w$.

A special case of this theorem (for $\k < \aleph_\w$) was proved already in \cite{Osipov} and \cite{Brian&Miller}. In \cite[Theorem 2.1]{Osipov}, it is shown that the three equivalent conditions of Theorem~\ref{thm:Banach} also imply that if $X$ is any Banach space of weight $\k$, then there is a continuous bijection from $X$ to the Hilbert cube $[0,1]^\w$.

The implication $(1) \Rightarrow (2)$ in Theorem~\ref{thm:Banach} holds for any $\k$. But the results of this section tell us nothing about whether the converse may hold for singular cardinals of countable cofinality. (It is proved in \cite[Section 2]{BCP} that if $\cf(\k) = \w$ then $\paar(\k^\w) \geq \cf[\k]^\w \geq \k^+$.)

\begin{question}
Suppose $\lambda$ is a singular cardinal with $\cf(\lambda) = \w$, and there is a continuous bijection $\lambda^\w \to \R$. Does this imply $\lambda \in \mathfrak{sp}(\text{\emph{\small Borel}})$?
\end{question}

\begin{question}
In particular, is it consistent that there is a continuous bijection $\w_\w^\w \to \R$, but that $\aleph_\w \notin \mathfrak{sp}(\text{\emph{\small Borel}})$?
\end{question}

\section{Other partition spectra}\label{sec:spectra}

Given a pointclass $\Gamma$ of sets, define
$$\specg = \set{|\P|}{\P \text{ is a partition of } \R \text{ into uncountably many sets in }\Gamma }\!.$$
Many of the known results concerning $\bspec$ listed in the introduction apply equally well to other pointclasses of sets. 
For example, both $\Spec$ and $\specod$ satisfy the analogues of statements $(2)$-$(6)$ in the introduction \cite{BBP,Blass}.
They also both satisfy the conclusion of the main result of the present paper. For $\specod$ this is relatively easy to see: the proof given in Section 3 generalizes readily to any pointclass $\Gamma$ that is closed under taking countable intersections and relative complements.

\begin{theorem}\label{thm:general}
Suppose $\Gamma$ is a pointclass closed under taking countable intersections and relative complements. Assuming $0^\dagger$ does not exist, if $\k$ is a singular cardinal with $\cf(\k) = \w$ and $\k \in \specg$, then $\k^+ \in \specg$.
\end{theorem}

We now show that the conclusion of this theorem applies to several familiar pointclasses that are not closed under taking countable intersections and relative complements.

\begin{theorem}
$\mathfrak{sp}(\text{\emph{\small Borel}}) = \mathfrak{sp}(\text{\emph{\small $\mathbf{\Sigma}^1_1$}}) = \mathfrak{sp}(\text{\emph{\small $\mathbf{\Pi}^1_1$}}) = \mathfrak{sp}(\text{\emph{\small $\mathbf{\Sigma}^1_2$}}) = \mathfrak{sp}(\text{\emph{\small $\aleph_1$-Suslin}})$.
\end{theorem}
\begin{proof}
Let $\Gamma$ be any of the latter four pointclasses in the statement of the theorem. 
Every set in $\Gamma$ can be be written as a union of $\aleph_1$ Borel sets (see \cite[Chapter 13]{Kanamori}).
This implies every set in $\Gamma$ can be partitioned into $\leq\!\aleph_1$ Borel sets.
Thus any partition $\P$ of $\R$ into sets in $\Gamma$ can be refined to a partition of $\R$ into Borel sets, by replacing each set in $\P$ with $\leq\!\aleph_1$ Borel sets; and if $\P$ is uncountable, any such refinement has size $|\P|$. 
Thus $\specg \sub \bspec$. The reverse inclusion is immediate: every Borel set is in $\Gamma$, so a partition of $\R$ into Borel sets is already a partition of $\R$ into sets in $\Gamma$.
\end{proof}

From this and Theorem~\ref{thm:main} it follows that the conclusion of Theorem~\ref{thm:general} applies when $\Gamma \in \{ \text{analytic, coanalytic, $\mathbf{\Sigma}^1_2$, $\aleph_1$-Suslin}\}$.

Next we show that the same is true for $\Gamma = {}$closed.
It is known that, consistently, $\Spec \neq \bspec$. This result can be attributed to Sierpi\'nski \cite{Sierpinski}, who showed (in our terminology) that $\min \!\left( \vphantom{f^2} \Spec \right) \geq \mathbf{cov}(\mathcal M)$, which means we need not have $\aleph_1 \in \Spec$. Nonetheless, we show here that $\Spec$ is a final segment of $\bspec$, which is enough to guarantee that Theorem~\ref{thm:general} holds for $\Gamma={}$closed.

Recall that the cardinal invariant $\mad_T$ denotes the least size of a partition of the Baire space $\w^\w$ into compact sets. (The notation ``$\mad_T$'' derives from the fact that $\mad_T$ is equal to the least size of a maximal infinite family of almost disjoint subtrees of $2^{<\w}$.)
It is known that $\dom \leq \mad_T$, and that this inequality is consistently strict \cite{Spinas}.

\begin{lemma}\label{lem:Spec1}
For any uncountable Polish space $X$,
\begin{align*}
\mathfrak{sp}(\text{\emph{\small closed}}) &= \set{\k > \aleph_0}{\text{there is a partition of } X \text{ into }\k \text{ compact sets}} \\ 
 &= \set{\k > \aleph_0}{\text{there is a partition of } X \text{ into }\k \text{ closed sets}} \\
  &= \set{\k > \aleph_0}{\text{there is a partition of } X \text{ into }\k \ F_\s \text{ sets}}.
\end{align*}
\end{lemma}
\begin{proof}
See \cite[Corollary 2.5]{BBP}.
\end{proof}

\begin{lemma}\label{lem:Spec}
$\mathfrak{sp}(\text{\emph{\small compact}}) = \mathfrak{sp}(\text{\emph{\small closed}}) = \mathfrak{sp}(\text{\emph{\small $F_\s$}}) = [\mad_T,\continuum] \cap \mathfrak{sp}(\text{\emph{\small Borel}})$.
\end{lemma}
\begin{proof}
By Lemma~\ref{lem:Spec1}, it suffices to show $\mathfrak{sp}(\text{\small compact}) = [\mad_T,\continuum] \cap \bspec$.

By Lemma~\ref{lem:Spec1} and the definition of $\mad_T$, $\mathfrak{sp}(\text{\small compact}) \sub [\mad_T,\continuum]$. And clearly every partition of $\R$ into compact sets is also a partition into Borel sets, so $\mathfrak{sp}(\text{\small compact}) \sub \bspec$. Hence $\mathfrak{sp}(\text{\small compact}) \sub [\mad_T,\continuum] \cap \bspec$.

Now suppose $\k \in [\mad_T,\continuum] \cap \bspec$, and let $\P$ be a partition of $\R$ into $\k$ Borel sets.
By a theorem of Lusin and Souslin \cite[Theorem 13.7]{Kechris}, every Borel subset of a Polish space is the bijective continuous image of a closed subset of $\w^\w$. 
In particular, for every $A \in \P$, there is a continuous bijection $f_A: F_A \to A$, where $F_A$ is a closed subset of $\w^\w$.
Every closed subset of $\w^\w$ is Polish, so for each $A \in \P$ there are two possibilities: (1) $A$ is countable, thus so is $F_A$, or (2) $A$ and $F_A$ are uncountable, and (by Lemma~\ref{lem:Spec1}) there is a partition of $F_A$ into $\mad_T$ compact sets. 
Let $\P_1 = \set{A \in \P}{A \text{ is countable}}$ and $\P_2 = \P \setminus \P_1$. For each $A \in \P_2$, fix a partition $\mathcal Q_A$ of $A$ into $\mad_T$ compact sets. Finally, observe that
$$\set{\{x\}}{x \in A \text{ for some }A \in \P_1} \cup \set{f_A[X]}{A \in \P_2 \text{ and }X \in \mathcal Q_A}$$
is a partition of $\R$ into $\k$ compact sets. 
\end{proof}

\begin{corollary}
Assume $0^\dagger$ does not exist. if $\k$ is a singular cardinal with $\cf(\k) = \w$ and $\k \in \mathfrak{sp}(\text{\emph{\small closed}})$, then $\k^+ \in \mathfrak{sp}(\text{\emph{\small closed}})$.
\end{corollary}

In addition to $\Spec$, $\bspec$, and $\specod$, which have been studied extensively, one may also find in the literature some interest in $\specgd$.
Most notably, Shelah and Fremlin prove in \cite{ShelahFremlin} that $\min \!\left( \vphantom{f^2} \specgd \right) \geq \mathbf{cov}(\mathcal M$) 
(which implies, in particular, there may be no partition of $\R$ into $\aleph_1$ $G_\dlt$ sets). 
This raises the intriguing possibility that an analogue of Lemma~\ref{lem:Spec} might hold for $\specgd$.

\begin{question}
Is $\mathbf{cov}(\mathcal M) = \min \!\left( \vphantom{f^2} \mathfrak{sp}(\text{\emph{\footnotesize $G_\dlt$}}) \right)$?
\end{question}

\begin{question}\label{q:?}
Is $\mathfrak{sp}(\text{\emph{\footnotesize $G_\dlt$}})$ a final segment of $\bspec$?
\end{question}

\begin{question}
If the answer to Question~\ref{q:?} is negative, is it nonetheless the case that Theorem~\ref{thm:general} holds with $\Gamma = G_\dlt$?
\end{question}



\end{document}